\documentclass[a4paper,reqno]{amsart}

\usepackage{hyperref}

\usepackage{amssymb,amsxtra,amsfonts,dsfont}

\usepackage{graphicx} \usepackage{amsmath} \usepackage{amssymb}

\usepackage[svgnames]{xcolor}

\usepackage[color,matrix,arrow]{xy}

\input xy \xyoption{all}

\newtheorem{lem}{Lemma}[section]

\newtheorem{thm}[lem]{Theorem}

\newtheorem{pro}[lem]{Proposition}

\newtheorem{cor}[lem]{Corollary}

\newtheorem{rem}[lem]{Remark}

\newcommand{\slrw}[1]{\stackrel{#1}{\longrightarrow}}

\newcommand{\lrw}{\longrightarrow}

\newcommand{\LL}{\Lambda}

\newcommand{\xa}{\alpha}
\newcommand{\xb}{\beta}
\newcommand{\xc}{\gamma}

\newcommand{\caDb}{\mathcal D^b}

\newcommand{\caM}{\mathcal M}

\newcommand{\caN}{\mathcal N}

\newcommand{\RHom}{\mathbf{R}\strut\kern-.2em\operatorname{Hom}\nolimits}

\newcommand{\zZ}{{\mathbb Z}}

\newcommand{\tV}{\widetilde{V}}

\newcommand{\tL}{\widetilde{\LL}}

\newcommand{\dt}{\Delta}

\newcommand{\dtL}{\Delta{\LL}}
\newcommand{\dtsL}{\Delta_{\sigma}{\LL}}
\newcommand{\dtnL}{\Delta_{\nu}{\LL}}

\newcommand{\olL}{\overline{\LL}}

\newcommand{\udgrm}{\underline{\mathrm{grmod}}}

\newcommand{\GG}{\Gamma}

\newcommand{\trho}{\tilde{\rho}}
\newcommand{\trhos}{\tilde{\rho}_{\sigma}}

\newcommand{\otsuGz}[1]{^{\otimes #1}}
\newcommand{\otsGz}{\otimes}

\newcommand{\tQ}{\widetilde{Q}}
\newcommand{\tQs}{\widetilde{Q}^{\sigma}}

\newcommand{\olS}{\overline{S}}

\newcommand{\olQ}{\overline{Q}}

\newcommand{\olrho}{\overline{\rho}}

\newcommand{\hhom}{\mathrm{Hom}}
\newcommand{\hhm}{\mathrm{Hom}}

\newcommand{\Ext}{\mathrm{Ext}}

\newcommand{\uhom}[1]{\underline{\mathrm{Hom}}}
\newcommand{\uExt}[1]{\underline{\mathrm{Ext}}}

\newcommand{\ohom}[1]{\overline{\mathrm{Hom}}}

\newcommand{\Ker}{\mathrm{Ker}\,}

\newcommand{\Img}{\mathrm{Im}\,}

\newcommand{\dmn}{\mathrm{dim}\,}

\newcommand{\arr}[2]{\begin{array}{#1}#2\end{array}}

\newcommand{\eqqc}[2]{\begin{equation}\label{#1}#2\end{equation}}

\newcommand{\eqqcn}[2]{\[ #2 \]}
\newcommand{\eqqcnn}[2]{$ #2 $}

\begin{document}

\title{On Trivial Extensions and Higher Preprojective Algebras}
\address{Jin Yun Guo\footnote{This work is supported by Natural Science Foundation of China \#11671126, and the Construct Program of the Key Discipline in Hunan Province}, LCSM( Ministry of Education), School of Mathematics and Statistics\\ Hunan Normal University\\ Changsha, Hunan 410081, P. R. China}
\email{gjy@hunnu.edu.cn}

\subjclass{Primary {16G20}; Secondary{16G70, 16S37}}

\keywords{ twisted trivial extension;  preprojective algebra; Koszul algebra; $n$-homogeneous algebra; $\tau$-slice algebra
}

\begin{abstract}
In this paper, we show that for a Koszul $n$-homogeneous algebra $\LL$,  the quadratic dual of certain twisted trivial extension is the $(n+1)$-preprojective algebra of its quadratic dual, that is, $ (\dtnL)^{!,op} \simeq\Pi( \LL^{ !, op })$.
This is applied to the $\tau$-slice algebras of stable $n$-translation algebras and gives a noncommutative version of Bernstein-Gelfand-Gelfand correspondence for such algebras.
\end{abstract}

\maketitle

\section{Introduction}
Preprojective algebra play an important role in representation theory of algebras \cite{gp79,bgl87, dr80, cb99} and many other fields of mathematics such as quantum groups, quiver varieties,  cluster algebras \cite{lu91, n94, gls06}.
Iyama and Oppermann introduced the $(n+1)$-preprojective algebra in studying higher representation\cite{io13}, which is  the $0$-th homology of Keller＊s derived (n+1)-preprojective DG algebra\cite{k11}.
Preprojective algebras and higher preprojective algebras are also important in non-commutative algebraic geometry\cite{cb99,mm11}.

Martinez-Villa has found that the quadratic dual of a preprojective algebra of a hereditary algebra is a self-injective algebra with vanishing radical cube \cite{m2}.
Relationship between higher preprojective algebras and trivial extensions is studied in \cite{mm11} for quasi-Fano algebras and is studied in \cite{gw18} for a class of $\tau$-slice algebras.

In this paper, we introduce $n$-homogeneous algebra (see Section 3).
By studying the quivers with relations as in \cite{th16}, we prove that for a Koszul $n$-homogeneous algebra $\LL$, the $(n+1)$-preprojective algebra $\Pi(\LL^{!,op})$ of its quadratic dual $\LL^{!,op}$ is isomorphic to the quadratic dual of a twisted trivial extension $\dtnL$ of $\LL$, that is, $\Pi(\LL^{!,op}) \simeq (\dtnL)^{!,op}$.

We recently find $\tau$-slice algebras of stable $n$-translation algebras are interesting in studying higher representation theory \cite{g16}.
We show that such algebras are Koszul $n$-homogeneous, so our mail theorem applied to this case as a generalization of the result obtained in \cite{gw18}.
When the $n$-translation algebra is the Koszul, we know that the quadratic dual of $\LL$ is a quasi $(n-1)$-Fano algebra (see \cite{mm11,gw18}) and the $(n+1)$-preprojective algebra is AS-regular algebra.
So we get  a  noncommutative version of Bernstein-Gelfand-Gelfand correspondence when the $(n+1)$-preprojective algebra is Noetherian (Theorem \ref{noncommBGG}).

The paper is organized as follows.
In Section 2, we recall basic notions in bound quivers, quadratic algebras and quadratic dual  needed in this paper.
In Section 3, we study the returning arrow quiver of $n$-homogeneous quiver and describe the trivial extension of an $n$-homogeneous algebra using returning arrow quiver and relations.
The relations of $(n+1)$-preprojective algebra of the quadratic dual of an $n$-homogeneous algebra are studied in Section 4, and the main theorem is proven in Section 5.
In the last section, we apply the main theorem to $\tau$-slice algebras of a stable $n$-translation algebra and obtain a noncommutative version of Bernstein-Gelfand-Gelfand correspondence.

\section{Preliminary}
Let $k$ be a field, and let $\LL = \LL_0 + \LL_1+\cdots$ be a graded algebra over $k$ with $\LL_0$ direct sum of copies of $k$ such that $\LL$ is generated by $\LL_0$ and $\LL_1$.
Such algebra is determined by a bound quiver $Q= (Q_0,Q_1, \rho)$ \cite{g16}.

Recall that a bound quiver $Q= (Q_0, Q_1, \rho)$ is a quiver with $Q_0$ the set of vertices, $Q_1$ the set of arrows and $\rho$ a set of relations.
The arrow set $Q_1$ is as usual defined with two maps $s, t$ from $Q_1$ to $Q_0$ to assign an arrow $\alpha$ its starting vertex $s(\alpha)$ and its ending vertex $t(\alpha)$.
The arrow set $Q_1$ in this paper is assumed to be locally finite in the sense that for each pair $i,j \in Q_0$, the set of arrows with $s(\alpha)=i$ and $t(\alpha)=j$ is finite.
We also write $s(p)=i$, $t(p)=j$ for a path in $Q$ from $i$ to $j$.
The relation set $\rho$ is a set of linear combinations of paths of length $\ge 2$, since we study graded algebra, it can be normalized such that each element in $\rho$ is a linear combination of paths of the same length starting at the same vertex and ending at the same vertex.

Let $S= \LL_0 = \bigoplus\limits_{i\in Q_0} k_i$, with $k_i \simeq k$ as algebras, and let $e_i$ be the image of the identity of $k$ under the canonical embedding of the $k_i$ into $S$.
Then $\{e_i| i \in Q_0\}$ is a complete set of orthogonal primitive idempotents in $S$ and $V= \LL_1 = \bigoplus\limits_{i,j \in Q_0 }e_j \LL_1 e_i$ as $S $-$ S $-bimodules.
Fix a basis $Q_1^{ij}$ of $e_j \LL_1 e_i$ for any pair $i, j\in Q_0$, take the elements of $Q_1^{ij}$ as arrows from $i$ to $j$, and let $Q_1= \cup_{(i,j)\in Q_0\times Q_0} Q_1^{ ij}.$
Conventionally, we write  $\otimes $ for $\otimes_{S}$.
Thus $V$ is the $S$-bimodule with $Q_1$ as a basis.
Let $Q_t$ be the set of the paths of length $t$ in $Q$ and let $kQ_t$ be the $k$-space spanned by $Q_t$.
When identifying $\xa_t\otimes \cdots \otimes \xa_1$ with path $p =\xa_t\cdots \xa_1$ in $Q$, then $V^{\otimes_{} t} = kQ_t$, and one gets an isomorphism of the tensor algebra $ T_{S}V = S \oplus V \oplus (V\otimes_{} V) + \cdots = \bigoplus_{t\ge 0} V^{\otimes_{} t }$ of $V$ over $S$ and the path algebra $kQ$: $T_{S}V\simeq kQ $ as algebras over $k$.
We will use tensor algebra $T_S V$ and path algebra $kQ$ alternatively in this paper.

There is canonical epimorphism from $kQ$ to $\LL$ with kernel $I$ contained in $kQ_2 + kQ_3 + \cdots$.
Choose a generating set $\rho$  of $I$, whose elements are linear combinations of paths of length $\ge 2$.
Then we have $I=(\rho)$ and $\LL \simeq kQ/(\rho)$, a quotient algebra of the path algebra $k Q$.
$\LL$ is also called the bound quiver algebra of bound quiver $Q= (Q_0,Q_1, \rho)$.
A path $p$ in $Q$ is called bound path if its image in $kQ/(\rho)$ is non-zero.
We usually use same notation for elements in the path algebra and its image in the bound quiver algebra when no confusion appears.

Let $Q= (Q_0, Q_1, \rho)$ be a bound quiver with quadratics relations.
In this case, the quotient algebra $\LL = kQ/(\rho)$ is called a {\em quadratic algebra.}
Write $DW=\hhom_k(W,k)$ for the dual space of a finitely dimensional $k$-space $W$.
By take dual bases locally, we have    $DS = k Q_0^{op} = \sum\limits_{i\in Q_0}k e_i^*$, $DV = k Q_1^{op} = \sum\limits_{\xa\in Q_1 } k \xa^*$, and identifying $D(V^{\otimes t}) $ with $(DV)^{\otimes t}$ via $\xc_t^*\otimes \cdots \otimes \xc_1^*(\xa_1\otimes \cdots\otimes \xa_t) = \xc_t^*(\cdots(\xc_1^*(\xa_1)\cdots)\xa_t)$.
So $(DV)^{\otimes t}$ is a dual space of  $V^{\otimes t}$ with $\{p^* =\xa_1^* \otimes\cdots\otimes \xa_t^* | p=\xa_t\otimes\cdots\otimes \xa_1\in Q_t\} $ a dual basis of $Q_t$.
We will need the following easy lemma for the transitional matrix of dual basis.
\begin{lem}\label{dualtransitionalmatrix}
Let $v_1,\cdots,v_m $ and  $w_1,\cdots,w_m$  be bases in a vector space $W$ with the transitional matrix $P$, and let $v^*_1,\cdots,v^*_n $ and  $w^*_1, \cdots, w^*_n$ be respectively  their dual bases in $DW$with the transitional matrix $T$.
Then $T^{-1}= P^{t}$.
\end{lem}

The quadratic dual quiver $Q^{op,\perp} = (Q^{op}_0, Q^{op}_1, \rho^{op, \perp})$ of $Q$ has the same vertex set $Q^{op}_0=Q_0$, has the local dual basis  $Q^{op}_1 =\{ \xa^*:j \to i|\xa:i\to j \in Q_1\} $ of $Q_1$ as arrow set, and has a basis $\rho^{op,\perp}$ in the orthogonal subspace $(k\rho)^{\perp} \in k Q^{op}_2$ of $k \rho$ in the dual space $kQ^{op}_2$ as relations.
Identifying the opposite quiver of $Q^{op,\perp}$ with $Q$, denoted by $Q^{\perp} = (Q_0, Q_1, \rho^{\perp})$  the opposite quiver of the quadratic dual quiver.
Set $\delta_{x,y} = \left\{\arr{ll}{1 & x=y\\0& x\neq y}\right.$, we have now  \eqqc{dualdef}{\xc(\xa)= \delta_{\xa,\xc} \mbox{ and }(\xc_1\otimes \cdots \otimes \xc_t )(\xa_1 \otimes \cdots \otimes \xa_t)= \xc_1(\xa_1)\cdots\xc_t(\xa_t).}
In this case, $(x,y)=x(y)$ for $x,y \in kQ_2$ defines a non-degenerated bilinear form on $kQ_2$.
The algebra $\LL^{!,op} = k Q^{\perp} /(\rho^{\perp})$ is called the {\em quadratic dual algebra} of $\LL$.

\medskip

$n$-translation quivers and $n$-translation algebras are introduced and studied in \cite{g16}.
With Nakayama permutation as the translation $\tau$, the bound quiver of a graded self-injective algebra of Lowey length $n+2$ is a stable $n$-translation quiver \cite{g12, g16}.
In fact, stable $n$-translation quivers are exactly the bound quivers of graded self-injective algebras of Lowey length $n+2$.
Recall that an algebra is called an {\em $n$-translation algebra}, if its bound quiver is an $n$-translation quiver and there is a $q \in \mathbb N \cup \{\infty \}$ such that it is $(n+1,q)$-Koszul \cite{g16,bbk02}.

Complete $\tau$-slices and $\tau$-slice algebras are introduced for stable $n$-translation quiver in \cite{g12}.
Let $\olQ=(\olQ_0, \olQ_1, \olrho)$ be a acyclic stable $n$-translation quiver with $n$-translation $\tau$, and assume that $\olQ$ has only finite many $\tau$-orbits.
A full subquiver $Q$ is called a  {\em complete $\tau$-slice} of $\olQ$ if it satisfies the following conditions:
(a). for each vertex $v $ of $\olQ$, the intersection of the $\tau$-orbit of $v$ and the vertex set of $Q$ is a single-point set;
(b). $Q$ is convex in the sense that for each path $p: v_0 \to v_1 \to \ldots \to v_t$ of  $\olQ$ with $v_0$ and $v_t$ in $Q$, the whole path $p$ lies in $Q$.
When normalizing the relations such that they are linear combinations of paths with the same starting vertex and the same ending vertex, then $$\rho = \{x = \sum\limits_p a_p p \in \olrho| s(p), t(p) \in Q_0 \} \subseteq \olrho,$$
and we regard a complete $\tau$-slice as the bound quiver $Q=(Q_0,Q_1,\rho)$.

The algebra $\LL$ defined by a complete $\tau$-slice $Q$ in  $\olQ$ is called a {\em $\tau$-slice algebra} of the bound quiver $\olQ$.
If $\olQ$ is the bound quiver of an algebra $\olL$, we also say that $\LL$ is a {\em $\tau$-slice algebra} of $\olL$.

If $\olQ$ define an $n$-translation algebra $\olL$, $\LL$ is a quadratic algebra, and its quadratic dual $\GG= {\LL}^{!,op}$ is called {\em dual $\tau$-slice algebra} of $\olQ$.

\medskip

\section{Trivial extension of $n$-homogeneous algebra}
We study the quiver and relation of the twisted trivial extensions of an $n$-homogeneous algebra

A bound quiver $Q =(Q_0,Q_1,\rho)$  is called {\em $n$-homogeneous} if the relations can be chosen homogeneous and all the maximal bound paths have the same length $n$.
We also call the (graded) algebra defined by an  $n$-homogeneous quiver an {\em $n$-homogeneous algebra.}

Recall that the trivial extension $\LL\ltimes M$ of an algebra $\LL$ by a $\LL$-bimodule $M$ is the algebra defined on the vector spaces $\LL\oplus M$ with the multiplication defined by $(a,x)(b,y)=(ab, ay+xb)$ for $a,b\in \LL$ and $x,y\in M$.
$\dtL = \LL\ltimes D\LL$ is called the  trivial extension of $\LL$.

Let $\sigma$ be a graded automorphism of $\LL$, that is, an automorphism that preserves the degree of homogeneous element.
Let $M$ be a $\LL$-bimodule.
Define the twist $M^{\sigma}$ of $M$ as the bimodule with $M$ as the vector space.
The left multiplication is the same as $M$, and the right multiplication is twisted by $\sigma$, that is,  defined by $x\cdot b=x\sigma(b)$ for all $x \in M^{\sigma}$ and $b \in \LL$.
Define the twisted trivial extension $ \dtsL= \LL \ltimes D \LL^{\sigma}$ to be the trivial extension of $\LL$ by the twisted $\LL$-bimodule $D \LL^{\sigma}$.

Let $\LL$ be an $n$-homogeneous algebra with bound quiver $Q =(Q_0, Q_1, \rho)$, let $\caM$ be a maximal linearly independent set of maximal bound paths.
Assume that $\LL=\LL_0 + \LL_1 + \cdots +\LL_n$ is an $n$-homogeneous algebra and that $\caM$ is a basis of $\LL_n$.
Let $\caM_t\subseteq Q_t$ be a maximal linearly independent set of bound paths of length $t$, then it is a basis of $\LL_t$ and we may take $\caM_n =\caM$.
Let $\caM^*_t= \{p^*| p\in \caM_t\}$ be the dual basis of $\caM_t$ in the dual space $D\LL_t$, then $D\LL$ is generated by $D\LL_n$ as a $\LL$-bimodule.
As an algebra, $\dtsL$ is generated by $\LL_0$ and $\LL_1 + D \LL_n$.
Let $\tilde{\LL}_0=\LL_0$, $\tilde{\LL}_t =\LL_t + D\LL_{n+1-t} $ for $t=1,\cdots, n$ and $\tilde{\LL}_{n+1} = D\LL_0$.
We see $\dtsL=\tilde{\LL}_0+\tilde{\LL}_1+\cdots + \tilde{\LL}_n+ \tilde{\LL}_{n+1}.$

Define {\em the returning arrow quiver} $\tQ= (\tQ_0,\tQ_1,\trho)$ with $\tQ_0 = Q_0$, $\tQ_1 = Q_1\cup Q_{1,\caM}$, with $Q_{1,\caM} = \{\beta_{p}: t(p) \to s(p)|  p\in \caM\}$.
So $\tQ$ is obtained by adding a arrow in the reversed direction to each maximal bound path in $Q$.

Recall that $S = \bigoplus\limits_{i\in Q_0} k e_i$ is a semisimple algebra, $V= \bigoplus\limits_{\xa\in Q_1} k \xa$ is the  $S$-bimodule spanned by the arrows in $Q$.
Let $V_{\caM}  = k Q_{\caM,1}$ be the $S$-bimodule spanned by the arrows in $Q_{\caM}$.
Let $\tV = V + V_{\caM} $, then $k \tQ = T_S \tV$.
The natural epimorphism $\mu$ from $kQ $ to $\LL$ is extended to an epimorphism $\mu_{\sigma} : k\tQ \to \dtsL$ by assigning $\beta_p$ to $(0,p^*)$ for each $p\in \caM$.
Let $\trho_{\sigma}$ be a generating set of $\Ker \mu_{\sigma} $, then $\dtsL \simeq k\tQ/(\trho_{\sigma})$, and $\trho_{\sigma}$ is a relation set for $\dtsL$.

We have the following proposition (see also Proposition 2.2 of \cite{fp02}).

\begin{pro}\label{quiverTE}
If $Q=(Q_0,Q_1,\rho)$ is the bound quiver of a graded algebra $\LL$ such that maximal bound paths have the same length $n$, then there is a relation set $\trhos$ such that $\tQs = (\tQ_0, \tQ_1, \trhos)$ is the bound quiver of the twisted trivial extension  $\dtsL$ of $\LL$.
\end{pro}

By replace trivial extension with twisted trivial extension, the quiver remains the same, but the relations change.
Write $\tQ=\tQs$ when $\sigma$ is identity, or when no confusion will appear.
If $Q$ is admissible $(n-1)$-translation quiver, the relation set has a nice description in \cite{g16}.

Assume that $\LL$ is  $n$-homogeneous algebra such that $\dtL$ is quadratic.
Then $\LL$ is quadratic since $\trho \cap V\otimes V$ spans the same ideal $(\rho)$ as $\rho$, and we may assume that $\rho$ is quadratic.
Now we determine the relation set $\trho$ in Proposition \ref{quiverTE} for the trivial extension $\dtL$.
We need only to determine its intersection with $\tV\otimes \tV$ since $\dtL$ is quadratic.
We can take it as a basis of the kernel of the linear map $$\mu|_{\tV\otimes \tV }:\tV\otimes \tV \to \tL_2 = \LL_2 + D\LL_{n-1}. $$
Since $\tV\otimes \tV = V\otimes V +  V\otimes V_{\caM} + V_{\caM} \otimes V+ V_{\caM} \otimes V_{\caM} $, and $\mu(V\otimes V) =\LL_2$, we have $\rho \subseteq \Ker \mu$ and $k\rho = \Ker \mu|_{V\otimes V}$.
On the other hand $\mu(V_{\caM} \otimes V_{\caM} ) =0$ implies  $V_{\caM} \otimes V_{\caM} \subseteq \Ker \mu $.
Write $\rho_{\caM} =\{\beta_{p}\otimes \beta_{p'} |p', p \in \caM, t(p)=s(p')\}$, it is a basis of $V_{\caM} \otimes V_{\caM} )$, and $ \rho_{\caM}\subseteq \trho$.

Now let $\caN = V\otimes V_{\caM} + V_{\caM} \otimes V$, we have $\mu (\caN) = D\LL_{n-1}$.
Let $\rho_{0} $ be a basis of $\Ker \mu |_{\caN}$.
The relation set for $\dtL$ is given in the following lemma.
\begin{lem}\label{relationsTE} If $\LL$ is $n$-homogeneous and $\dtL $ is quadratic, then
$$\trho = \rho \cup \rho_{\caM}\cup \rho_{0}.$$
\end{lem}

Now let $\sigma$ be a graded automorphism of $\LL$, then $\sigma$ is determined by its action on $\LL_0$ and $\LL_1$.
Consider the bound quiver $\tQs$ of $\dtsL$.
It has the same quiver $\tQ=(\tQ_0,\tQ_1)$ as $\dtL$.
The relations in $\rho$ concern only elements in $\LL$, so they remain unchanged.
Since $(D\LL)^2=0$ in $\dtsL$, so $\rho_{\caM}$ is in the relation set of $\dtsL$.
The remaining relations are obtained by twisting those in $\rho_{0}$ by $\sigma$.
In fact, for $\xb_t, \xb_t'\in Q_{1,\caM}, \xa_t,\xa_t'\in Q_1$, we have  $\sum\limits_{t}b_t {\alpha}_t \cdot {\beta}_t + \sum\limits_{t} b'_t{\beta}'_{t} \cdot {\xa}'_t = 0$ in $\dtsL$ if and only if $ \sum\limits_{t} b_t{\xa}_t {\beta}_t + \sum\limits_t b'_t {\beta}'_{t} \sigma({\xa}'_t) =0$ in $\dtL$ for $b_t,b'_t \in k$.
Let $$\arr{lll}{\rho_{\sigma, 0} &=& \{\sum\limits_{\xa\in Q_1, p\in \caM_n} b_{\xa,p}\xa \otimes {\beta}_p + \sum\limits_{\xa\in Q_1, p\in \caM_n} b'_{\xa,p}  {\beta}_{p}  \otimes \sigma({\xa}) |\\&& \quad \sum\limits_{\xa\in Q_1, p\in \caM_n} b_{\xa,p}\xa \otimes {\beta}_p + \sum\limits_{\xa\in Q_1, p\in \caM_n} b'_{\xa,p}  {\beta}_{p}  \otimes {\xa}  \in \rho_{0} \}  .}$$
Let $\trho^{\sigma} = \rho \cup \rho_{\caM} \cup\rho_{\sigma, 0}$ and let $\tQ^{\sigma} = (\tQ_0,\tQ_1,\trho^{\sigma})$.
We have the following corollary.

\begin{cor}\label{relationsTEtwisted}Assume that $\LL$ is $n$-homogeneous and $\dtL $ is quadratic, and $\sigma$ is an graded automorphism of $\LL$. Then $\tQs$ is the bound quiver of $\dtsL$.
\end{cor}

\section{$(n+1)$-preprojective algebra of  dual  $n$-homogeneous algebra}

Let $\GG$ be a finite dimensional algebra of global dimension $n$, the $(n+1)$-preprojective algebra of $\GG$ is defined as the tensor algebra $$\Pi(\GG)= T_{\GG} \Ext_{\GG}^n(D\GG,\GG)\\ \simeq T_{\GG} \Ext_{\GG^e}^n (\GG,\GG^e) ,$$ where $\GG^e = \GG\otimes \GG$ is the envelope algebra of $\GG$.
In \cite{th16}, a description of $(n+1)$-preprojective algebra for Koszul algebra using quiver with relations is given.

Now we study the bound quiver of the $(n+1)$-preprojective algebra of the quadratic dual of an $n$-homogeneous algebra, we call such algebra a {\em dual $n$-homogeneous algebra}.

Throughout this section, we assume that $\LL $ is a Koszul $n$-homogeneous algebra, and that $\GG =\LL^{!,op}$ is its quadratic dual.
Since $\LL$ is a Koszul algebra with Loewy length $n+1$, $\GG$ is Koszul  of global dimension $n$ (see Theorem 2.6.1 and Subsection 2.8 of \cite{bgs96}).

Note that $\LL_1= V =kQ_1$ and $\GG_1 = (DV)^{op}_1 = kQ_1 $, by identifying the arrow set $Q_1$ with its dual basis.
The bound quiver of $\GG$ is $Q^{\perp} =  (Q_0, Q_1, \rho^{ \perp })$ and $\GG = T_S V /(\rho^\perp)$, where $\rho^{\perp}$ can be taken as a basis of the orthogonal subspace of $\rho$ in $kQ_2 \simeq V \otimes V$.
Let $R$ be the space spanned by the relation $\rho^{\perp}$ in $kQ_2 $.
Let $K^{l}_{l} = \cap_{t} V\otsuGz{t} \otsGz R \otsGz V \otsuGz{l-2-t} $, then it follows from the Subsection 2.8 of \cite{bgs96} that  $K_t^t = \LL_t$.
It follows from Theorem 2.6.1 of \cite{bgs96} and \cite{h89} that we have a Koszul bimodule complex,
\eqqc{KcomplexGG}{0 \lrw \GG\otsGz \LL_n \otsGz\GG\slrw{f=f_n} \GG\otsGz \LL_{n-1} \otsGz\GG \slrw{f_{n-1}} \cdots \slrw{f_1} \GG\otsGz \GG \slrw{f_0} \GG \lrw 0.}
For a path $p$ of length $t$, write $\xa_l(p)=\xa$, $\chi_l(p)=q$ if $p=\xa q$ for $\xa\in Q_1$ and $q\in Q_{t-1}$, and write $\xa_r(p)=\xa'$, $\chi_r(p)=q'$ if $p=q'\xa'$ for $\xa'\in Q_1$ and $q'\in Q_{t-1}$.
The bimodule map  $f_t: \GG\otimes \LL_t \otimes \GG \to \GG\otimes \LL_{t-1} \otimes \GG$ is defined by \eqqcnn{nmapf}{f_t (p) = \xa_l(p) \otimes \chi_l(p) +(-1)^{t} \chi_r(p) \otimes \xa_r(p)}.

Since $\GG$ is Koszul, \eqref{KcomplexGG} is the projective resolution of $\GG$ as $\GG^e$ module.
Following \cite{th16}, by applying $\hhom_{\GG^e }(\quad, \GG^e)$ we get
\eqqcn{extformula}{\arr{l}{\cdots\lrw \hhom_{\GG^e}(\GG \otimes \LL_{n-1} \otimes \GG, \GG^e ) \\ \qquad \slrw{f^*} \hhom_{\GG^e}(\GG \otimes \LL_{n} \otimes \GG, \GG^e ) \lrw \Ext_{\GG^e}^n(\GG,\GG^e) \lrw 0}.}
Note that $\GG \otimes D\LL_{n-1} \otimes \GG $  is  a $\GG$-bimodule generated by $D\LL_{n-1}$, so $\Img f^*$ is a $\GG$-bimodule generated by $f^*(D\LL_{n-1})$.

$f^*$ induces a bimodule map, denoted again by the same notation,  $f^*: \GG\otimes D\LL_{n-1} \otimes \GG \to \GG\otimes D\LL_{n} \otimes \GG$, which is the composition of the bimodule maps \small\eqqcn{ncompmaps}{\arr{lll}{\GG\otimes D\LL_{n-1} \otimes \GG  &\slrw{\Psi} & \hhom_{\GG^e}( \GG\otimes \LL_{n-1} \otimes \GG, \GG^e)\\ & \slrw{\hhom_{\GG^e}(f,\GG^e)} & \hhom_{\GG^e}( \GG\otimes \LL_{n} \otimes \GG, \GG^e) \slrw{\Phi} \GG\otimes D\LL_{n} \otimes \GG ,} }\normalsize
where $\Psi:  x\otimes q^* \otimes y \to \{u \otimes q' \otimes v \to q^*(q') uy \otimes xv \}$, $\hhom_{\GG^e}(f,\GG^e):  \psi \to \psi \circ f $ for $\psi$ from $\GG\otimes \LL_{n-1} \otimes \GG$ to $\GG^e$ and $\Phi:  \phi \to \sum\limits_{p,t}  \phi_{1,t}(p) \otimes p^* \otimes \phi_{2,t}(p) $ for $x,y,u,v \in \GG$, $q,q'\in \LL_{n-1}$, $\phi \in \hhom_{\GG^e}(\GG \otimes \LL_n \otimes \GG, \GG^e) $ such that $\phi(z)= \sum\limits_{t}\phi_{2,t}(z)\otimes \phi_{1,t}(z)$ for $z\in \GG\otimes \LL_{n} \otimes \GG$.

So $\Ext^n_{\GG}(D\GG,\GG) \simeq \Ext_{\GG^e}^n(\GG,\GG^e)\simeq (\GG \otimes D\LL_n \otimes \GG)/\Img f^*$ as $\GG$-bimodules.

By definition, \eqqcnn{}{\Pi(\GG) = T_{\GG}\Ext_{\GG}(D\GG,\GG) = T_{\GG}(( \GG \otimes D\LL_n \otimes \GG) /( f^*(D\LL_{n-1}))) }.
Since $D\LL_n=  V_{\caM}$ and $\tV =V +V_{\caM}$, one see easily that we have a natural map $\theta: T_S \tV =k\tQ \to \Pi(\GG)$, which extends the natural epimorphism from $kQ \to \GG$.
Let $\rho_{\caM^\perp} = \{f^*(q^*_r)| q_r\in \caM_{n-1}\}$, then $\Ker \theta = (\rho^\perp \cup \rho_{\caM^\perp})$ is the $k\tQ$-bimodule generated by $\rho^\perp \cup \rho_{\caM^\perp}$.
Let $\trho^* =\rho^\perp \cup  \rho_{\caM^\perp}$, we get a description of $(n+1)$-preprojective algebra $\Pi(\GG)$ of the dual $n$-homogeneous algebra $\GG$ with bound quiver $Q^\perp =(Q_0,Q_1,\rho^{\perp})$ as in \cite{th16}.
\begin{pro}\label{preproalg}{$ \Pi(\GG) \simeq k\tQ/(\trho^*)$.}
\end{pro}

Now we determine $\rho_{\caM^\perp}$, the image of $\caM_{n-1}$ under the map $f^*$.

Assume that $\caM_{n-1} = \{q_1, \cdots, q_{m}\}$ is a basis of $\LL_{n-1}$ with each $q_r$ a bound path of length $n-1$.
Take $q_r\in \caM_{n-1}$, a path from $i$ to $j$.
For each $\xa \in  Q_1 e_j $ and $\xa' \in e_i Q_1$, $\xa q_r, q_r\xa' \in \LL_n$.
So there are $a^{\xa', q_r}_p, a^{q_r,\xa'}_p  \in k$ for $p\in \caM_n$ such that \eqqc{relpqxa}{ \xa q_r =\sum\limits_{p \in \caM_n e_i } a^{\xa, q_r}_p p \mbox{ and } q_r \xa'=\sum\limits_{p \in e_j\caM_n } a^{q_r,\xa'}_p p,} and thus $p^*(\xa q_r) =a^{\xa, q_r}_p$ and $p^*(q_r \xa') =a^{q_r,\xa'}_p$ since $\caM^*$ is the dual basis of $\caM$ in $D\LL_n$.

For each $q\in \caM_{n-1}$, $q^* \in D\LL_{n-1} \subseteq \GG \otimes D\LL_{n-1} \otimes \GG$, let $\zeta_q= f^*(q^*)= f^*(1\otimes q^* \otimes 1)= \Phi\hhom_{\GG^e}(f,\GG^e) \Psi (1\otimes q^* \otimes 1) =\Phi \Psi (1\otimes q^* \otimes 1) f.$
Since $\caM^*_{n-1}$ is a generating set of $\GG$-bimodule $\GG \otimes D\LL_{n-1} \otimes \GG $, such elements generating $\Img f^*$.

We have the following lemma.
\begin{lem}\label{fqr}
For each $q_r \in e_j \caM_{n-1} e_i$,
\eqqc{fqformula}{ \zeta_{q_r}= \sum\limits_{\xa\in Q_1 e_j, p \in  \caM_n e_i} a^{\xa, q_r}_{p} p^*\otimes \xa +(-1)^n \sum\limits_{\xa\in Q_1,p \in e_j \caM_n}  a^{ q_r,\xa}_{p} \xa\otimes p^*.}
\end{lem}
\begin{proof}
For each bound path $p\in e_{j}\LL_n$, we have $\Psi (1\otimes q_r^* \otimes 1) f(1\otimes p \otimes 1)=(-1)^n q_r^*(\chi_r(p)) e_j  \otimes \xa_r(p)$ and for each $p\in \LL_n e_i$, we have $\Psi (1\otimes q_r^* \otimes 1) f(1\otimes p \otimes 1)= \xa_l(p) \otimes e_i q_r^*(\chi_l(p))$.

This implies that $\zeta_{q_r} = \Phi\Psi (1\otimes q_r^* \otimes 1)f  $ is in the subspace of $ \GG_1\otimes D\LL_{n} + D\LL_n\otimes \GG_{1}\simeq \caN$ spanned by the set  $\{\xa \otimes p^*|\xa\in Q_1, p\in e_j \caM_n \} \cup \{p^* \otimes \xa |\xa\in Q_1, p\in \caM_n e_i \} $.
Write \eqqcn{fqrex}{\zeta_{q_r}=\sum\limits_{\xa\in Q_1, p\in \caM_n e_i} c_{\xa,p}(q_r) p^*\otimes\xa + \sum\limits_{\xa\in Q_1, p\in e_j \caM_n} c_{p,\xa}(q_r) \xa\otimes p^*.}
This is an element in the dual space of $\LL_1\otimes \LL_{n} + \LL_n\otimes \LL_{1}$, which is also isomorphic to $\caN$  via \eqref{dualdef}.

Embed $\LL_n$ in $\LL_1\otimes \LL_{n}$ by sending a path $p'$ to $\xa_l(p') \otimes p'$, and in $ \LL_n\otimes \LL_{1}$ by sending a path $p'$ to $p' \otimes \xa_r(p')$.
So $\zeta_{q_r}$ acts on it by  $(\xa\otimes p^*) (p') = (\xa\otimes p^*)(\xa_r(p') \otimes p') = \xa^*(\xa_r(p'))\cdot p^*(p')$ and $(p^*\otimes\xa) (p') = ( p^*\otimes\xa)( p'\otimes\xa_l(p')) = \xa^*(\xa_l(p'))\cdot p^*(p')$.
Then we have that $\zeta_{q_r} (p') = \zeta_{q_r} (\xa_r(p') \otimes  p')  = c_{\xa_r(p'),p'}(q_r) $.
Especially, we have that $\zeta_{q_r} (p') = c_{p',\xa_r(p')}(q_r) = (-1)^n$ if $p' = q_r\xa_r(p')$ and $\zeta_{q_r} (p')  = 1$ if $p' = \xa_l(p')q_r$.

Now we determine $c_{\xa', p}(q_r)$ for $p\in \caM_n$ with $p^*( q_r\xa') \neq 0$.

Write $e_j\caM_n = \{p_t\}_t$, it is a basis of $e_j \LL_n$.
For each $t_0$ with $a^{q_r, \xa'}_{p_{t_0}} \neq 0$, by replacing $p_{t_0}$ with $q_r\xa'$, we obtain a new basis $\{{p'}_t\}_t$ of $e_j\LL_n$ with $p_t =p'_t$ for $t\neq t_0$ and ${p'}_{t_0} = q_r \xa'$.
Then $\zeta_{q_r}(q_r\xa') = (-1)^n e_j \otimes \xa' $ and the coefficient of $\xa'\otimes (q_r\xa')^*$ in $\zeta_{q_r}$ is $(-1)^n$.
On the other hand, we have  ${p'}_t^*({p'}_{t'}) = \delta_{t,t'}$, so the coefficient of $\xa'\otimes {p'}_t^*$ is zero for $p'_t \neq q_r \xa'$.
Let $\zeta_{q_r}=\sum\limits_{\xa\in Q_1}\sum\limits_{t} c'_{\xa, p_t}(q_r) {p'}_t^*\otimes\xa + \sum\limits_{\xa\in Q_1}\sum\limits_{t} c'_{p_t,\xa}(q_r) \xa\otimes {p'}_t^*$, then $c'_{p'_{t_0},\xa'}(q_r) =(-1)^n$ and $c'_{p'_{t},\xa'}(q_r) =0$ for $t\neq t_0$.
This implies that $c_{p_{t_0},\xa'}(q_r) = \sum\limits_{t}a^{\xa',q_r}_{p_t} c'_{\xa',p'_{t}}(q_r) = (-1)^n a^{q_r,\xa'}_{p_{t_0}} $, by Lemma \ref{dualtransitionalmatrix}.
Do this for all $t_0$ with $a^{q_r,\xa'}_{p_{t_0}}\neq 0$ and for all  $\xa'\in Q_1$, we get $c_{p,\xa}(q_r)= (-1)^n a^{q_r,\xa}_{p}$ for $\xa \in e_i Q_1$ and $p \in e_j \caM_n$.

Similarly, we have that $c_{\xa, p}(q_r)=  a^{\xa, q_r}_{t}$ for $\xa \in Q_1 e_j$ and $p\in \caM_n e_i$.
This proves our lemma.
\end{proof}

So we have the following description of $\rho_{\caM^\perp}$.
\begin{cor}\label{rltion} Assume \eqref{relpqxa}, then
\small\eqqcn{relation_prepro}{\arr{lll}{\rho_{\caM^\perp} &=& \{
\sum\limits_{\xa\in Q_1 e_j, p\in \caM_n e_i}a^{\xa, q}_{p} p^*\otimes \xa +(-1)^n \sum\limits_{\xa\in e_i Q_1, p\in e_j\caM_n} a^{ q,\xa}_{p} \xa\otimes p^* \\ && \qquad |i,j\in Q_0, q \in e_j \caM_{n-1} e_i\}.}}\normalsize
\end{cor}
In fact, the relation set is minimal in the following sense.
\begin{lem}\label{lid}
$\rho_{\caM^\perp}$ is a linearly independent set.
\end{lem}
\begin{proof}
We see from the proof of Lemma \ref{fqr} that for suitably chosen basis $\caM_n$ of $\LL_n$, if $p_t\in \caM_n$ then $\xa_r(p_t) \otimes p_t^*$ will only appears in $\zeta_{q_r}$ with coefficient $(-1)^n$ for $q_r=\chi_r(p_t)$, and $p_t^*\otimes \xa_l(p_t) $ will only appears in $\zeta_{q_r}$ with coefficient $1$ for $q_r=\chi_l(p_t)$.
Thus the set $\rho_{\caM^\perp} = \{\zeta_{q_r}| q_r \in e_j \caM_{n-1} e_i\}$ is a linearly independent set.
\end{proof}

Since $\rho_{\caM^\perp}$ is indexed by a basis of $\LL_{n-1} $, which have the same dimesion as $D\LL_{n-1}$, as an corollary, we get the following proposition.
\begin{pro}\label{dimeq}
$\dmn_k D\LL_{n-1} = \dmn_k k\rho_{\caM^{\perp}}$.
\end{pro}

\section{Trivial extension and $(n+1)$-preprojective algebras}
Now let $\LL$ be a Koszul $n$-homogeneous algebra with bound quiver $Q=(Q_0,Q_1,\rho)$ such that $\dtL$ is Koszul.
Let $\GG= \LL^{!,op}$ be its quadratic dual, then $\GG$ is defined by the bound quiver $Q^{\perp} = (Q_0,Q_1,\rho^{\perp})$.
Let $\caM_{n-1}, \caM_{n}$ be the maximal linearly independent set of bound paths of length $n-1$ and of length $n$, respectively.
Define $\nu$ as the automorphism of $\LL$ induced by the linear map $\nu: \xa \to (-1)^n\xa$ on $kQ_1$.
Let $\tQ = (Q_0, \tQ_1)$ be the returning arrow quiver of $Q$ with $\tQ_1 = Q_1 \cup Q_{\caM,1}$,  then $k\tQ \simeq T_S \tV $ for  $S=\bigoplus_{i\in Q_0} k$,  $V=kQ_1$, $V_{\caM} =k Q_{\caM,1}$ and $\tV = V \oplus V_{\caM}$. 

We have epimorphisms $\mu_{\nu}:  T_S \tV \to \dtnL$, and $\theta:  T_S \tV \to \Pi(\GG)$, and the  kernels of these maps are generated by their quadratic subspaces in  $\tV\otimes \tV$, respectively.
Recall that $\tV\otimes \tV = V\otimes V\oplus \caN\oplus V_{\caM}\otimes V_{\caM}$ and $\caN = V\otimes V_{\caM}\oplus V_{\caM}\otimes V$.
By identify $\tV$ with its dual spaces, paths of length $2$ form a basis  in $\tV\otimes \tV$ which is its own dual basis as in \eqref{dualdef}.
The subspaces $V\otimes V, \caN=V\otimes V_{\caM}\oplus V_{\caM}\otimes V, V_{\caM}\otimes V_{\caM}$ are pairwise orthogonal.

By Lemma \ref{relationsTE} and Corollary \ref{relationsTEtwisted},
$\Ker \mu_{\nu} =k \rho \oplus k \rho_{\caM}\oplus k \rho_{\nu,0}$, where $k \rho = \Ker \mu_{\nu} \cap V\otimes V$, $k \rho_{\caM} = V_{\caM}\otimes V_{\caM}$ and $\Ker \mu_{\nu}\cap \caN =k\rho_{\nu,0} $.
By Corollary \ref{relationsTEtwisted},  \eqqc{mapnu}{\arr{lll}{ \rho_{\nu, 0} &=& \{\sum\limits_{\xa\in Q_1, p\in \caM_n} b_{\xa,p}\xa \otimes {\beta}_p + (-1)^n \sum\limits_{\xa\in Q_1, p\in \caM_n} b'_{\xa,p}  {\beta}_{p}  \otimes {\xa} |\\&& \qquad  \sum\limits_{\xa\in Q_1, p\in \caM_n} b_{\xa,p}\xa \otimes {\beta}_p + \sum\limits_{\xa\in Q_1, p\in \caM_n} b'_{\xa,p}  {\beta}_{p}  \otimes {\xa}  \in \rho_{0}\}
}}

By Proposition \ref{preproalg}, $\Ker \theta = k\rho^{\perp} \oplus k \rho_{\caM^{\perp}}$, with $ k\rho^{\perp} = \Ker \theta \cap V\otimes V$ and $k\rho_{\caM^{\perp}} = \Ker \theta \cap \caN$.
$k\rho$ and $k \rho^{\perp}$ are orthogonal subspaces in $V\otimes V$.
Now we prove that $k\rho_{\nu,0}$ and $k\rho_{\caM^{\perp}}$ are orthogonal subspaces in $ \caN$.

We first prove the following lemma.
\begin{lem}\label{orthogonal}
$e_i\rho_{\caM^\perp} e_j$ and $\Ker \mu_{\nu}|_{e_i\caN e_j} = e_i k\rho_{\nu,0}e_j $ are orthogonal.
\end{lem}

\begin{proof}
For each $w \in e_j\rho_{\caM^\perp} e_i$, we have $q_r\in e_j \caM_{n-1} e_i$ such that $w=\zeta_{q_r}$ has the form \eqref{fqformula}.
For $z\in \Ker \mu_{\nu}|_{e_i\caN e_j}$, we have \eqqcnn{nzcoe}{z=\sum\limits_{\xa\in Q_1 e_j,p\in \caM_n  e_i } (-1)^n b_{t, \xa} \beta_{p_t} \otimes \xa + \sum\limits_{\xa\in e_i Q_1, p\in e_j \caM_n} b_{\xa, p} \xa\otimes \beta_{p_t} } by \eqref{mapnu}.
So by \eqref{dualdef}, \eqqcn{ortho}{\arr{lll}{z (w)  &= &
(-1)^n(\sum\limits_{\xa\in e_i Q_1, p \in e_j \caM_n} b_{p, \xa} a^{\xa, q_r}_{p} +  \sum\limits_{\xa\in Q_1,p \in e_j \caM_n}  b_{\xa, p} a^{ q_r,\xa}_{p} ),
}}
since $(\xa \otimes \beta_{p}) (\beta_{p'}\otimes \xa') =0 $, $(\beta_{p}\otimes \xa)(\xa' \otimes \beta_{p'}) =0 $, $(\beta_{p}\otimes \xa)(\beta_{p'}\otimes \xa') =\delta_{\xa,\xa'}\delta_{p,p'} =(\xa \otimes \beta_{p}) (\beta_{p'}\otimes \xa') $ whenever $\xa \otimes \beta_{p}\neq 0$ and $\beta_{p}\otimes \xa \neq 0$.

On the other hand, since $z\in \Ker \mu$, we have that \eqqcn{naa}{0 = \mu_{\nu}(z)(q_r)=\sum\limits_{\xa\in e_j Q_1, p\in \caM_n e_i} b_{ \xa, p}  p^*(\xa q_r ) + \sum\limits_{\xa\in Q_1 e_{i}, p\in e_j \caM_n} b_{p, \xa}  p^*(q_r\xa ),} where $ p^*( \xa q_r) = a^{ q_r, \xa }_{p} $ and $ p^*(q_r\xa ) = a^{\xa, q_r}_{p}$.
This proves that $z(w)=0$, and thus $e_i\rho_{\caM^\perp} e_j$ and $\Ker \mu_{\nu}|_{e_i\caN e_j}$ are orthogonal.
\end{proof}

So we have that $\rho_{\nu,0}=\cup_{i,j\in Q_0} e_j\rho_{\nu,0}e_i  $ and $\rho_{\caM^{\perp}} = \cup_{i,j\in Q_0} e_j\rho_{\caM^{\perp}}e_i  $ are also orthogonal.
Note that $ D\LL_{n-1} + k\rho_{\nu, 0} = \caN $ as direct sum of $k$-spaces, and we have that $\dmn_k k\rho_{\caM^{\perp}} = \dmn_k D\LL_n$ by  Proposition \ref{dimeq}.
By comparing the dimensions, we get that  $\caN = k\rho_{\nu, 0} + k\rho_{\caM^{\perp}}$ as direct sum of $k$-spaces.
So $k\rho_{\caM^+} = \bigoplus_{i,j\in Q_0} e_i k\rho_{\caM^+} e_j $ and we have the following lemma.

\begin{lem}\label{orthogonal}
$k\rho_{\nu,0}$ and $k \rho_{\caM^{\perp}}$ are orthogonal subspaces in $\caN$.
\end{lem}

Thus we get our  main theorem of this section.

\begin{thm}\label{mainthm}
Assume that $\LL$ is a Koszul $n$-homogeneous algebra. Then $$\Pi( \LL^{ !, op }) \simeq (\dtnL)^{!,op}.$$
\end{thm}
\begin{proof}
By Proposition \ref{quiverTE}, $\dtnL \simeq k\tQ/(\trho)$,  and  by Lemma \ref{relationsTE}, $\trho =\rho\cup \rho_{\nu,0}\cup \rho_{\caM}$. Proposition \ref{preproalg}, $\Pi( \LL^{ !, op }) \simeq k\tQ/(\trho^*)$ with $\trho^*=\rho^{\perp}\cup \rho_{\caM^{\perp}}$.
By Proposition \ref{preproalg}, $\Pi( \LL^{ !, op }) \simeq k\tQ/(\trho^{\perp})$ with $\trho^{\perp}=\rho^{\perp}\cup \rho_{\caM^{\perp}}$.

Since $\rho$ and $\rho^{\perp}$ are orthogonal subspaces in $V\otimes V$, $\rho_{\caM} = V_{\caM}\otimes V_{\caM}$,  we get that  $k\trho = k\rho\oplus k \rho_{\caM} \oplus k\rho_{\nu,0} $ and $k\trho^{\perp} = k\rho^{\perp}\oplus k \rho_{\caM^{\perp}} $ are orthogonal subspaces in $\tV\otimes \tV$ by Lemma \ref{orthogonal}.

This proves that $\trho^* = \trho^\perp$, and thus $\Pi( \LL^{ !, op }) \simeq k\tQ/(\trho^\perp)\simeq (\dtnL)^{!,op}$.
\end{proof}

\begin{rem}{\em

This is a known result for the path algebras of acyclic quivers.
For a stable $1$-translation algebra, a $\tau$-slice algebra is a radical squared zero algebra with acyclic quiver, say $Q$.
So its quadratic dual (dual $\tau$-slice algebra) is the path algebra of the same acyclic quiver $Q$.
It is proven in \cite{m2} that  the quadratic dual of the preprojective algebra of a path algebra of quiver $Q$ is self-injective algebra with vanishing radical cube and thus by \cite{cx09}, is a twisted trivial extension of the radical squared zero algebra with quiver $Q$.
In this case, the returning arrow quiver is exactly the double quiver of $Q$.
}\end{rem}

\section{$\tau$-slice algebras of a stable $n$-translation algebra}

Now we apply our main theorem to the $\tau$-slice algebras of stable $n$-translation algebra.

We first show that $\tau$-algebras are characterized by $n$-homogeneous quivers.
Assume that $\LL$ is an acyclic $\tau$-slice algebra with bound quiver $Q = (Q_0, Q_1,\rho)$ of an acyclic stable $n$-translation algebra $\LL$.
Assume that $Q$ is a $\tau$-slice of a stable $n$-translation quiver $\olQ$ and the algebra $\olL$  defined by $\olQ$ is an $n$-translation algebra.

\begin{lem}\label{maxpaths}
Let $Q$ be a complete $\tau$-slice of acyclic stable $n$-translation quiver $\olQ$, then $Q$ is $n$-homogeneous.
\end{lem}
\begin{proof}
Since the relations in an $n$-translation quiver can be chosen homogeneous, so do the relations in a complete $\tau$-slice $Q$.

Let $p$ be a maximal bound path in $Q$, then its length is $< n+1$, by the definition of a complete $\tau$-slice.
If the length $l$ of $p$ is $< n$, there is a path $q$ from $t(p)$ to $\tau^{-1} s(p)$ of length $n+1-l$ in $\olQ$ such that $qp \neq 0$ in $\olQ$, since $\olQ$ is a stable translation quiver.
Especially, there is an arrow $\beta: t(p)\to j$ in $\olQ$ such that $\beta p \neq 0$.
This contradicts the maximality of $p$ if $j\in Q_0$.
If $j \not\in  Q_0$, then $\tau j \in Q_0$, and there is a path $q'$ in $\olQ$ of length $n+1-l$ from $\tau j$ to $s(p)$, such that $q'p \neq 0$ in $\olL$.
But $Q$ is a complete $\tau$-slice in $\olQ$, and thus $q'$ is a non-trivial path in $Q$ and $q'p \neq 0$ in $\LL$, contradicts the maximality of $p$ in $Q$.

This proves that the bound quiver $Q$ is $n$-homogeneous.
\end{proof}

\medskip

For an algebra $\LL$ satisfying the condition of Proposition \ref{quiverTE}, we can generalize our quiver construction $\zZ|_{n-1} Q$ in \cite{g16}.
Such quiver construction is an $n$-analog of the classical $\zZ Q$ construction. 

For simplicity, we assume that  $\dtL$ is quadratic.
Take vertex set $(\zZ|_{n-1} Q)_0 =\{(i , t)| i\in Q_0, t \in \zZ\}$, arrow set $(\zZ|_{n -1}Q)_1  = \zZ \times Q_1 \cup \zZ \times \caM^{op} = \{(\alpha,t): (i,t)\longrightarrow (j,t) | \alpha:i\longrightarrow j \in Q_1, t \in \zZ\} \cup \quad\{(\beta_p , t): (j, t) \longrightarrow (i, t+1) | p\in \caM, s(p)=i,t(p)=j  \}$ and relation set $\rho_{\zZ|_{n-1} Q} =\zZ \rho\cup \zZ \rho_{\caM} \cup \zZ\rho_0$, where
$\zZ \rho =  \{\sum_{s} a_s (\xa_s,t)\otimes (\xa'_s,t) |\sum_{s} a_s \xa_s\otimes \xa'_s \in \rho, t\in \zZ\}$, $\zZ \rho_{\caM} =  \{(\beta_{p'},t+1) \otimes (\beta_p ,t)| \beta_{p'}\otimes \beta_{p}\in \rho_{\caM}, t\in \zZ\}$ and $\zZ \rho_0 = \{ \sum_{s'} a_{s'} (\beta_{p'_{s'}},t+1)\otimes  (\xa'_{s'}, t) + \sum_{s} b_s (\xa_s,t)\otimes  (\beta_{p_s} ,t)| \sum_{s'} a_{s'} \beta_{p'_{s'}}\otimes \xa'_{s'} + \sum_{s} b_s \xa_s,t\otimes \beta_{p_s} \in \rho_0 , t\in \zZ\}$.

Similar to Proposition 5.5 of \cite{g16}, we have the following realization of $\zZ|_{n-1} Q$.

\begin{pro}\label{0nap:extendible1}
Assume that $\LL$  is an algebra whose bound quiver is acyclic $n$-homogeneous and $\dtL$ is quadratic.
Then the smash product $\dtL\#  k \zZ^*$ is a self-injective algebra with bound quiver $\zZ|_{n-1} Q $, where $\dtL$ is graded by taking elements in the dual basis of $\caM$ in $D\LL_n$ as degree $1$ generators.
\end{pro}

If $Q$ is acyclic, it is a complete $\tau$-slice in $\zZ|_{n-1} Q$, so $\LL$ is a $\tau$-slice algebra of $\dtL\#  k \zZ^*$.
In such case, $\LL$ is also called a {\em $\tau$-slice algebra} of $\dtL$.

Combine with Lemma \ref{maxpaths}, we have the following Proposition.

\begin{pro}\label{0nap:slicenhom}
Let $\LL$  be an algebra with acyclic bound quiver $Q$.
Then $\LL$ is a $\tau$-slice algebra if and only if $Q$ is $n$-homogeneous.
\end{pro}

As an immediate consequence, using Theorem 5.3 of \cite{g16}, we have the following result.
\begin{pro}\label{0nap:slicentrans}
Let $\LL$  be an algebra with acyclic bound quiver $Q$.
Then $\LL$ is a $\tau$-slice algebra of an $n$-translation algebra if and only if $Q$ is $n$-homogeneous and $\dtL$ is almost Koszul.
\end{pro}

\medskip

Now we prove the Koszulity of $\tau$-slice algebras.

\begin{pro}\label{sliceKoszul}
Acyclic $\tau$-slice algebras of stable $n$-translation algebras are Koszul.
\end{pro}
\begin{proof}
Let $\LL$ be an acyclic $\tau$-slice algebra.
Assume that $\olQ$ is the bound quiver of an acyclic stable $n$-translation algebra $\olL$ and the bound quiver $Q$ of $\LL$ is a complete $\tau$-slice of $\olQ$.
Regard $Q$ as a full sub-quiver of $\olQ$.
For each integer $t$, write $M(t)$ for the shift module of degree $t$ for $M$, that is, the graded module $M(t) = \sum\limits_{s\in \zZ} M(t)_s $ with $M(t)_s =M_{t+s}$.
Denote by $u_i(j)$ the length of paths from $i$ to $j$ in $\olQ$, for a vertex $j$ such that there is a path from $i$ to $j$, it is uniquely determined since $\olQ$ is acyclic and $\olL$ is graded.
Let $e= \sum\limits_{i\in Q_0} e_i$, we have $\LL = e\olL e$.
Since $\olL$ is $n$-translation algebra, we have $q \ge 2$ or $q=\infty$, such that for each $i\in Q_0$, we have a projective resolution
\eqqc{projresolL}{\cdots \lrw \sum\limits_{j\in \olQ_0, u_i(j) =t } (\olL e_j (t))^{a_{i,j}} \slrw{\phi_t} \cdots \lrw \olL e_i(0) \lrw \olL_0 e_i \lrw 0,} for the graded simple $\olL$-module $\olS_i \simeq \olL_0 e_i$ concentrated in degree $0$, with $\Ker \phi_q = \sum\limits_{j\in \olQ_0, u_i(j) =q } (\olL_{n+1} e_j (q))^{a_{i,j}} $ if $q$ is finite.
$\olL_{n+1} e_j (t) \simeq \olL_0 e_{\tau^{-1} j} (t+n+1)$ is semisimple.
Since $Q$ is convex in $\olL$, we have $\hhm_{\olL}(\olL e, \olL e_j (t)) = 0$ if $j\not\in Q_0$ and $\hhm_{\olL}(\olL e, \olL e_j (t)) \simeq e\olL e_j \simeq \LL e_j(t) $ if $j\in Q_0$.
Apply $\hhm_{\olL}(\olL e, -)$ on \eqref{projresolL} for the projective $\olL$-module $\olL e$, one get and exact sequence
\eqqcn{projresolLL}{\cdots \lrw \sum\limits_{j\in Q_0, u_i(j) =t } (\LL e_j (t))^{a_{i,j}} \slrw{\phi'_t} \cdots \lrw \LL e_i(0) \lrw \LL_0 e_i \lrw 0,} for the graded simple $\LL$ module $\LL_0 e_i$.
We have $\Ker \phi'_q = 0$ when $q$ is  finite since $Q$ is a complete $\tau$-slice of $\olQ$, and $e_{\tau^{-1}j} \not\in Q_0$ for any $j\in Q_0$.
This shows that each graded simple $\LL$-module has a linear resolution, and thus $\LL$ is Koszul.
\end{proof}

It follows from 2.10 of \cite{bgs96} that for a stable $n$-translation algebra, both $\tau$-slice algebras and dual $\tau$-slice algebras are Koszul algebras.

\medskip

Now we have the following version of our main for $\tau$-slice algebras.

\begin{thm}\label{cormain}
Assume that $\LL$ is a $\tau$-slice algebra of an acyclic stable $n$-translation algebra. Then $$\Pi( \LL^{ !, op }) \simeq (\dtnL)^{!,op}.$$
\end{thm}
\begin{proof} By Proposition \ref{0nap:slicentrans}, $\LL$ is $n$-homogeneous, and by Proposition \ref{sliceKoszul}  it is Koszul.
So by Theorem \ref{mainthm}, we have $$\Pi( \LL^{ !, op }) \simeq (\dtnL)^{!,op}.$$
\end{proof}

This is a generalization of our result obtained in \cite{gw18}, in which case $\LL$ is a special truncation of $\olQ$ and $\GG= \LL^{!,op}$ is quasi $(n-1)$-Fano algebra, using a different approach.

\medskip

Let $\GG =\LL^{!,op}$ be the quadratic dual of $\LL$.
It is known that $\Pi(\GG) = \Pi(\LL^{ !, op }) $ is AS-regular algebra when it is Koszul\cite{m2,mm11}.
Denote by $\mathrm{qgr}\Pi(\LL)$ the non-commutative projective scheme of $\Pi(\GG)$ \cite{az94}.
Then by Theorem 4.14 of \cite{mm11}, we have an equivalence of the triangulated categories $\caDb (\mathrm{qgr}\Pi(\GG)) \cong \caDb(\GG)$ when $\Pi(\GG)$ is Noetherian.
By Theorem 2.12.6 of \cite{bgs96}, we have an equivalence of the triangulated categories $\caDb(\LL)\cong \caDb(\GG)$ in this case.
By Corollary 6.11 of \cite{g12} $\caDb(\LL)\cong \udgrm \dtnL$, the stable category of graded $\dtnL$ modules.
So we get the following noncommutative version of  Bernstein-Gelfand-Gelfand correspondence \cite{bgg78}.

\begin{thm}\label{noncommBGG} Assume that $\GG$ is a dual $\tau$-slice algebra.
If $\Pi(\GG)$ is Noetherian, then we have an equivalence of the triangulated categories
$$\caDb (\mathrm{qgr}\Pi(\GG)) \cong\udgrm \dt_{\nu} \GG^{!,op}. $$
\end{thm}

In fact, we have the following picture relate a $\tau$-slice algebra, its twisted trivial extension, its quadratic dual and the higher preprojective algebra of the quadratic dual.
$$\xymatrix@C=2.0cm@R1.5cm{ \LL \ar@{<->}[rr]^-{\txt{       quadratic dual} }\ar@/^/[d]^-{\txt{ twisted\\ trivial\\ extension}}& &\GG \ar@/^/[d]^-{\txt{$(n+1)$-preprojective \\algebra}} \\ \dtnL \ar@{<->}[rr]^-{\txt{       quadratic dual}}\ar@/^/[u]^-{\txt{$\tau$-slice\\ algebra}}& &\Pi(\GG) }$$
In certain cases, $\GG$ is realized as a direct summand of the Beilinson algebra of $\Pi(\GG)$ \cite{mm11,gw18}.


{}

\end{document}